\documentclass{amsart}

\usepackage{amsthm}
\usepackage{cite}
\usepackage{enumerate}
\usepackage{hyperref}
\usepackage{graphicx}
\usepackage[labelfont=bf]{caption}
\usepackage{subcaption}
\usepackage{amsmath}
\usepackage{amssymb}

\newtheorem{theorem}{Theorem}[section]
\newtheorem{lemma}[theorem]{Lemma}
\newtheorem{prop}[theorem]{Proposition}
\newtheorem{corollary}[theorem]{Corollary}

\newtheorem{claim}{Claim}

\theoremstyle{definition}

\theoremstyle{remark}

\newtheoremstyle{case}{}{}{}{}{}{:}{ }{}
\theoremstyle{case}

\newcommand{\defeq}{\mathrel{\mathop:}=}

\newcommand{\dt}{$(m_1,\cdots, m_N) \times (t_1, \cdots, t_N) \in \mathbb{Z}_{\geq 0}^N \times \mathbb{Z}^N$}

\numberwithin{equation}{section}

\begin{document}

\title{Algorithm for filling curves on surfaces}

\author{Monika Kudlinska}

\address{School of Mathematics, University of Bristol, Woodland Road, Bristol BS8 1UG, United Kingdom}

\email{monika.kudlinska@bristol.ac.uk}

\maketitle

\begin{abstract}
Let $\Sigma$ be a compact, orientable surface of negative Euler characteristic, and let $h$ be a complete hyperbolic metric on $\Sigma$. A geodesic curve $\gamma$ in $\Sigma$ is \emph{filling} if it cuts the surface into topological disks and annuli. We propose an efficient algorithm for deciding whether a geodesic curve, represented as a word in some generators of $\pi_1(\Sigma)$, is filling. In the process, we find an explicit bound for the combinatorial length of a curve given by its Dehn-Thurston coordinate, in terms of the hyperbolic length. This gives us an efficient method for producing a collection which is guaranteed to contain all words corresponding to simple geodesics of bounded hyperbolic length.
\keywords{filling curves \and hyperbolic surface \and Dehn-Thurston coordinates \and geodesics}
\end{abstract}


\section{Introduction}

Let $\Sigma$ be a compact, orientable surface of negative Euler characteristic. Recall, a curve $\gamma : S^1 \to \Sigma$ is said to be in minimal position, if it is self-transverse, and the number of self-intersections is minimal over all curves freely homotopic to $\gamma$. A curve $\gamma$ in minimal position is \emph{filling} if $\Sigma\, - \gamma$ is a collection of topological disks and annuli, such that each annulus is homotopic to a boundary component of $\Sigma$. The main result of this note is the following:
\begin{theorem} \label{main} There exists a polynomial time algorithm to determine whether a curve $\gamma$ in $\Sigma$ is filling. 
\end{theorem}

The input of the algorithm in Theorem~\ref{main} is a word of length $L$ in some fixed generating set $X$ of $\pi_1(\Sigma)$. We show that our algorithm terminates in $O(L^{2N + 2})$ time, where $N$ denotes the complexity of the surface $\Sigma$. If $\Sigma$ has genus $g$ and $n$ boundary components, recall that its complexity is defined to be $N= 3g - 3 + n$.

We point out that there exists another algorithm for determining whether a curve is filling, as given in the PhD thesis \cite{Chris}. The basic idea of \cite{Chris} is to construct a curve with minimal self-intersection, corresponding to a word in a generating set of $\pi_1(\Sigma)$. The algorithm then gives a way of detecting whether the complementary regions of the curve are (possibly punctured) disks. As will be explained in the following paragraph, our approach is much different and unlike the above, we get estimates for the running time of our algorithm. 

Let us fix a complete hyperbolic metric on $\Sigma$. From here on, we identify each curve $\gamma$ in $\Sigma$ with its free homotopy class in $\Sigma$, and define its length $l(\gamma)$ to be the length of the unique geodesic in that class. The intersection between two curves $\gamma$ and $\gamma'$ is taken to be the minimum number of transverse intersections between any two curves homotopic to $\gamma$ and $\gamma'$, respectively. One can easily see that a curve is filling, if and only if it intersects every simple curve in $\Sigma$. In fact, a sufficient condition for $\gamma$ to be filling is that it intersects every simple curve of length at most twice the length of itself (Lemma~\ref{length}). 

Our strategy to prove Theorem~\ref{main} will thus be as follows. Given a curve $\gamma$, we will construct a set containing all words in some generating set $X$ of $\pi_1(\Sigma)$ corresponding to \emph{simple} curves of length bounded by $2l(\gamma)$. We will then check whether each curve in our set intersects $\gamma$, thus determining whether $\gamma$ is filling. To that end, there exists a number of algorithms for calculating the intersection number of curves represented as words in $X$, see \cite{CohLus}, \cite{Lus} and \cite{Tan}. Most recently,  Despre and Lazarus \cite{DesLaz} have given an algorithm which runs in $O(L^2)$ time, where $L$ is a bound on the length of the words representing the curves.

In order to construct a set containing all simple words in $X$, we recall the \emph{Dehn-Thurston parametrisation} of simple curves. For a fixed pants decomposition $\mathcal{K} = \{K_i\}_{i=1}^{N}$ of $\Sigma$, the Dehn-Thurston coordinate of a simple curve $\alpha$ is defined to be the vector \[p(\alpha) := (m_1, \cdots, m_N)\times (t_1, \cdots, t_N) \in \mathbb{Z}_{\geq 0}^N \times \mathbb{Z}^N,\] where each $m_i$ is the intersection number of $\alpha$ with the pants curve $K_i$, and each $t_i$ is a `twisting parameter' which counts the number of times $\alpha$ traverses each curve $K_i$.  We define the \emph{combinatorial length} of a simple geodesic $\alpha$, to be the sum \[\l_p(\alpha) := \sum_{i=1}^N m_i + \sum_{i=1}^N \left|t_i\right|.\] Although it is easy to see that the combinatorial length of a curve is comparable to its hyperbolic length, our algorithm requires calculation of explicit bounds. We note that there exist various methods for obtaining such bounds, for instance by quantifying the proof of the Milnor-S\`{v}arc Lemma. Here we use a more direct approach:

\begin{prop} \label{comb length}
Fix a complete hyperbolic metric on $\Sigma$ to be so that each pants curve has length $\frac{9}{10}$. For any simple geodesic $\alpha$ in $\Sigma$, the combinatorial length of $\alpha$ satisfies \[l_p(\alpha) \leq 4l_h(\alpha). \]
\end{prop}

The final step of our algorithm is to write the curves as words in a generating set $X$ of $\pi_1(\Sigma)$. We construct a specific generating set $X$ which fits our purpose well, and which is closely related to the Dehn-Thurston coordinates (see Section~\ref{dict}). Given the bound from Proposition~\ref{comb length}, we can then construct the required set of simple words of bounded length, thus also proving the following proposition. We say a hyperbolic metric on $\Sigma$ is \emph{admissible}, if each pants curve in $\mathcal{K}$ has length at most $\frac{9}{10}$.

\begin{prop} \label{list} Let $\Sigma$ be a compact, orientable hyperbolic surface with an admissible hyperbolic metric. For any $L > 0$, there exists an explicit method of constructing the set $\mathcal{W}(L)$, which contains all words corresponding to simple curves of hyperbolic length at most $L$, and satisfies $\left|\mathcal{W}(L)\right| \leq 2^N {4L + 2N \choose 2N}$.
\end{prop}

This paper is organised as follows. In Section~\ref{2} we review the relevant background material, including the Dehn-Thurston coordinates, and explain the dictionary between the coordinates and word representation of curves. In Section~\ref{3} we prove the bound between hyperbolic and combinatorial lengths of simple curves from Proposition~\ref{comb length}. Finally in Section~\ref{4} we collect results about filling curves and prove Theorem~\ref{main}.

\subsection{Acknowledgements} I would like to express my deepest gratitude to Viveka~Erlandsson for her patience, expertise, and passion whilst supervising this undergraduate project. I would also like to thank Juan Souto for his valuable comments on a draft of this paper, and the anonymous referee for the suggested improvements. This work was partially supported by the London Mathematical Society Research Bursary Scheme, Grant Reference 17-18 56.


\section{Background}\label{2}

We describe the \emph{Dehn-Thurston} coordinates of multiarcs in $\Sigma$. Originally attributed to Dehn, the parametrisation was rediscovered by Thurston \cite{Thurston}. We present here a brief overview of the coordinates. For a more detailed account see \cite{PenHar}.

\subsection{Preliminaries}\label{prelim}

Throughout this paper we let $\Sigma = \Sigma_{g,n}$ be a compact, oriented surface with genus $g$, and $n$ boundary components. Let $\partial \Sigma$ denote the boundary of $\Sigma$, and let $\{ \delta_1, \cdots, \delta_n\}$ be the set of connected components of $\partial \Sigma$. We assume that $\Sigma$ has negative Euler characteristic, and we quip $\Sigma$ with a complete hyperbolic metric $h$ such that the connected components of $\partial\Sigma$ (if any) are geodesics. We let $\tilde{\Sigma}$ denote the universal cover of $\Sigma$ which, as usual, we identify with a subset of the hyperbolic plane $\mathbb{H} = \left\{z \in \mathbb{C} \mid \text{Im}(z) > 0\right\}$. We will use the term curve to mean an immersion $\gamma : S^1 \to \Sigma$, and arc an immersion $\alpha: [0,1] \to \Sigma$ such that $\alpha(0), \alpha(1) \in \partial \Sigma$. We say a curve in $\Sigma$ is $\emph{essential}$ if it is not homotopic to a boundary component, nor to a point in $\Sigma$. An arc is essential if it cannot be homotoped into the boundary, relative its endpoints. We define a \emph{multiarc} in $\Sigma$ to be a finite collection of homotopy classes of simple curves and simple arcs in $\Sigma$, which are essential and pairwise disjoint. A \emph{multicurve} is a multiarc with no arc components. Recall that the homotopy class of any curve $\gamma$ in $\Sigma$ contains a unique geodesic. We let $l_h{(\gamma)}$ denote the length of that unique geodesic. If $\alpha$ is an arc, we write $l_h(\alpha)$ to mean the length of a shortest representative in the homotopy class, where the homotopy is relative to $\partial \Sigma$. For a multiarc $\Gamma = \sum_{i=1}^n{\gamma_i}$, we define its length $l_h{(\Gamma)}$ to be the sum $l_h{(\Gamma)} = \sum_{i=1}^n{l_h(\gamma_i)}.$ We define the \emph{(geometric) intersection number} of two curves $\alpha$ and $\beta$ to be \[\iota(\alpha, \beta) = \text{min}\{|\alpha' \cap \beta'| \mid \alpha' \sim \alpha, \beta' \sim \beta\},\]
Here $\alpha \sim \beta$  denotes the existence of homotopy between $\alpha$ and $\beta$, where the homotopy is relative to the boundary $\partial\Sigma$ if $\alpha$ and $\beta$ are arcs. Note that this definition extends naturally to multiarcs. 

We will need the following standard result from hyperbolic geometry (see \cite{Keen} and \cite{buscollar}). If $\gamma$ is a simple geodesic curve in a hyperbolic surface $\Sigma$, a \emph{collar} of width $w$ around $\gamma$ is the set $C(w) = \left\{ x \in \Sigma \mid d_h(x,\gamma) \leq w \right\}$. Let $w_{\gamma}$ be the largest $w$ for which the collar $C(w)$ is an embedded annulus in $\Sigma$. The $\emph{Collar Lemma}$ states that \[ \sinh(w_{\gamma}) \geq  1\, \big/ \sinh\left(\frac{l_h(\gamma)}{2}\right). \]
Moreover, for any collection of simple, pairwise disjoint geodesic curves $\{\gamma_i\}$ in $\Sigma$, the corresponding collars $C(\gamma_i, w_{\gamma_i})$ are pairwise disjoint \cite[Theorem~4.1.1]{Buser}.

Let $P$ denote a surface homeomorphic to a sphere with three disks removed, which we will refer to as a \emph{pair of pants}. For the remainder of this note, we fix a complete hyperbolic metric $h$ on $P$ to be such that each boundary component has length $\frac{9}{10}$. Elementary hyperbolic computations show that the length of each seam $s$, (the shortest arc joining any two distinct boundary components), in our metric on $P$ satisfies $l_h(s) \approx 3.06$ and the length of each mid $\nu$, (the shortest essential arc joining a boundary component to itself), satisfies $l_h(\nu) \approx 4.57$. We record these here for later. In what follows, we will refer to the two hexagonal regions in $P$ bounded by the seams and the boundary curves, as \emph{faces} of $P$.

\subsection{Dehn-Thurston coordinates}\label{2.4}

Fix a pants decomposition $\mathcal{K} = \{K_i\}_{i=1}^N$ of $\Sigma$, and let $\mathcal{P} = \{P_k\}_{k=1}^M$ be the corresponding set of pairs of pants. For each pants curve $K_i$, pick a closed subarc $w_i \subset K_i$ called the \emph{window} of the pants curve, and a point $p_i \in w_i$ called the \emph{marked point}. For each pair of pants $P_k \in \mathcal{P}$, and for every pair of (not necessarily distinct) marked points in the boundary of $P_k$, fix a shortest simple oriented arc that is essential in $P_k$, and whose endpoints are the marked points. The resulting set of arcs is called the set of \emph{canonical arcs} of $\Sigma$. For each index $k$, let $A_k$ denote the set of canonical arcs in the pair of pants $P_k$.

Given a multiarc $C$ in $\Sigma$, the Dehn-Thurston parameter $(m_1, \cdots, m_N)\times (t_1, \cdots, t_N) \in \mathbb{Z}^N_{\geq 0} \times \mathbb{Z}^N$ of $C$ is defined as follows. For each index $i$, let $m_i = \iota(C, K_i)$ be the intersection of $C$ with $K_i$. Consider the connected 1-complex in $\Sigma$ consisting of the pants curves and the canonical arcs. Fix $\epsilon > 0$, and isotope $C$ so that it is contained in the $\epsilon$-neighbourhood of the 1-complex. If $C$ does not intersect the pants curve $K_i$, set $t_i$ to be the number of components of $C$ in the $\epsilon$-neighbourhood of $K_i$. For each index $i$ fix the rectangle $R_i = w_i \times [-\epsilon, \epsilon]$, and let $C'$ be a representative of $C$ which satisfies $C' \cap \left(w_i \times \{t\}\right) = m_i$, for every $t\in [-\epsilon, \epsilon]$. If $m_i > 0$, let $c$ be the multiarc segment of $C'$ contained in the $\epsilon$-neighbourhood of $K_i$ . The parameter $\left|t_i\right|$ is defined to be half the minimum intersection of $c'$ with the two edges of $R_i$ perpendicular to $w_i$, over all arcs $c'$ homotopic to $c$, fixing endpoints. We set the sign of $t_i$ to be positive if some strand of $C$ travels to the right of the $\epsilon$-neighbourhood of $K_i$ (treated as an oriented annulus, with orientation induced from $\Sigma$), and negative otherwise.

It follows that every simple curve can be identified with a point in $\mathbb{Z}^N_{\geq 0} \times \mathbb{Z}^N$, and one can show that this point is unique. Conversely, a point in $\mathbb{Z}^N_{\geq 0} \times \mathbb{Z}^N$ corresponds to a Dehn-Thurston coordinate of a multicurve, provided that it satisfies a set of simple conditions. We will not need these here, however the interested reader is referred to \cite{PenHar}. We only note that it follows that the number of multicurves of combinatorial length at most $L$ is bounded by $2^N { 2N + L \choose L}$, which grows like $O(L^{2N})$.

\subsection{Dictionary between coordinates and words}\label{dict}

Let $\pi_1(\Sigma,p)$ denote the fundamental group of $\Sigma$ based at $p$, and without loss of generality pick $p$ to be a point from the set $\{p_i\}$ of marked points of the pants curves in $\Sigma$. Let $T$ be a spanning tree of the 1-complex in $\Sigma$ consisting of pants curves and canonical arcs (as above). Fix an orientation for each of the pants curve in $\mathcal{K}$. For each index $i$, let $a_i$ be the unique oriented path in the in the spanning tree $T$ from $p$ to $p_i \in K_i$. Define the oriented loop $\tilde{K_i} \defeq a_iK_i a_i^{-1}$ based at $p$. For each $k$ and every $l \in A_k$, let $\tilde{l}$ denote the corresponding oriented loop based at $p$, for some fixed orientation, and write $\tilde{A}_k = \{\tilde{l} \mid l\in A_k \}.$ Let $X = \{\tilde{K}_i\} \cup \bigcup \tilde{A}_k$, and note that this set generates $\pi_1(\Sigma,p)$.


Suppose $C$ is a multiarc in $\Sigma$. Recall that the Dehn-Thurston coordinate of $C$ is obtained by homotoping $C$ so that it is carried by the 1-complex consisting of pants curves and canonical arcs in $\Sigma$. Thus, given the Dehn-Thurston coordinate of $C$, it is possible to represent $C$ as a concatenation of canonical arcs and pants curves, that is $C = u_1\cdots u_n$ where each $u_l \in \{K_i\} \cup \bigcup A_k$. For each $u_l$ in the decomposition of $C$, let $\tilde{u}_l$ be the corresponding loop at $p$ (as defined above) and let $\tilde{C} = \tilde{u_1} \cdots \tilde{u_n}$ be the concatenation of these loops. Since the endpoints of consecutive arcs $u_l, u_{l+1}$ in $C$ coincide, we must have that the arc which connects the endpoint of $u_l$ to $p$ and the arc which connects $p$ to the start point of $u_{l+1}$ cancel out. Thus, \[\tilde{C} = \tilde{u_1} \cdots \tilde{u_k} = (a_{i_1}u_1a_{j_1}^{-1})(a_{i_2}u_2a_{j_2}^{-1}) \cdots (a_{i_k}u_ka_{j_k}^{-1}) = a_{i_1}(u_1u_2 \cdots u_k)a_{j_k}^{-1}.\] Hence, we can identify $C$ with the conjugacy class $ [ \tilde{u_1} \cdots \tilde{u_k}]$ in $\pi_1(\Sigma,p)$, and thus write it as a word in $X$ of length $l_p(C)$. As a result, we obtain a dictionary between the Dehn-Thurston coordinates, and words in generators $X$ of $\pi_1(\Sigma,p).$

For later use, we record here a bound for the hyperbolic length of a curve, in terms of the length of a word in $X$ which represents it. As before, we fix the hyperbolic metric on $\Sigma$ to be so that each pants curve has length $\frac{9}{10}$. From the calculations at the end of Section~\ref{prelim}, it follows that the length of each canonical arc joining two distinct pants curves is bounded by $3.1 + 2\frac{9}{10} < 5$, and the length of canonical arc joining the same boundary component is bounded by $5 + \frac{9}{10} < 6$. Recall that each edge of the spanning tree $T$ is a canonical arc of $\Sigma$. We define the length of $T$ to be the sum of the lengths of the canonical arcs which constitute its edges. It is clear that the spanning tree $T$ can only contain canonical arcs with distinct endpoints, and furthermore $T$ can contain at most 2 arcs from each pair of pants. Thus the length of $T$ is bounded by $10M$, where $M = 2g - 2 + n$ is the number of pairs of pants in $\Sigma$. Each generator in $X$ has length at most twice the length of $T$, plus the length of the longest canonical arc,  or pants curve. Hence, the length of each generator is bounded by $20M + 6 \leq 26M$. It follows that if $\gamma$ is any curve in $\Sigma$ which can be represented as a word of length $L$ in $X$, then $l_h(\gamma) \leq 26ML.$


\section{Bound for the combinatorial length of geodesics}\label{3}

In this section we prove Proposition~\ref{comb length} which relates the combinatorial length of a simple curve to its hyperbolic length. The main idea is to first prove bounds relating the combinatorial and hyperbolic lengths of a multiarc in a pair of pants. By applying the bound to segments of the curve in each pair of pants of the pants decomposition of $\Sigma$, we extend the result to a bound for a curve in the whole surface.

\subsection{Multiarcs in pairs of pants}


Fix a basis for the Dehn-Thurston parameters by taking the marked points $\{p_i\}_{i=1}^3$ in the boundary of $P$ to be such that they are contained in the same face of $P$, and the canonical arcs 
to be the \emph{shortest} essential arcs joining each pair of marked points. Given a multiarc $A$ and its Dehn-Thurston parametrisation \dt, recall that we defined the combinatorial length of $A$ to be the sum \[l_p(A) = \sum_{i=1}^Nm_i + \sum_{j=1}^N\left|t_j\right|.\]

\begin{prop}\label{6.1} For any simple connected arc $a$ in a pair of pants $P$ with endpoints contained in the set of marked points $\{p_i\}_{i=1}^3 \subset \partial P$, \begin{equation}\label{eq:0}l_p{(a)} \leq \frac{20}{9}l_h(a). \end{equation}

\end{prop}

Let $a$ be a simple, geodesic arc in $P$ with endpoints which coincide with the marked points $\{p_i\}_{i=1}^3$. Let $p(a) = (m_1, m_2, m_3) \times (t_1, t_2, t_3)$ be the Dehn-Thurston parametrisation of $a$. When $a$ is a canonical arc in $P$, we have that $t_i = 0$ for each $i$, and $\sum m_i = 2$. Thus $l_p(a) = 2 \leq l_h(s) \leq l_h(\nu)$, where $s$ is a seam of $P$ and $\nu$ a mid of $P$. Hence $l_p(a) \leq l_h(a)$.

\begin{claim}\label{diff} The bound \eqref{eq:0} holds for any simple, non-canonical arc $a$ in $P$ with distinct endpoints contained in the set $\{p_i\}_{i=1}^3 \subset \partial P$.

\end{claim}

\begin{proof}

Assume that $a$ has endpoints $a(0) = p_1 \in \delta_1$ and $a(1) = p_2 \in \delta_2$, the other cases can be treated analogously. Let $a^*$ be the shortest arc that is homotopic to $a$ (fixing endpoints), and which traverses only the boundary components $\delta_1$, $\delta_2$ and the seam $s$ connecting them. For each index $i$, let $\left|\tau_i \right|$ be the length of the subarc of $a^{*}$ which traverses the boundary $\delta_i$. We set $\tau_i$ to be positive if $a^{*}$ travels to the right of the boundary component, and negative otherwise.

We first observe that
\begin{equation}
\left| \tau_1 \right| + \left| \tau_2 \right| \geq \left|t_1\right| + \left|t_2 \right| - 1, \label{eq:1}
 \end{equation}
where $t_1, t_2$ are the twisting parameters from the Dehn-Thurston parametrisation of $a$. Indeed, the distance between the marked $p_i$ and the endpoint of the seam $s$ in $\delta_i$ is at most half the length of $\delta_i$, for $i=1,2$, and so \eqref{eq:1} follows from the definition of the twisting parameter.

Next, we show that
\begin{equation}
\label{eq:2}
2l_h(a) \geq l_h(s) + \left|\tau_1\right|l_h(\delta_1) + \left|\tau_2\right
|l_h(\delta_2).
\end{equation}
Since $l_h(\delta_1) = l_h(\delta_2) = \frac{9}{10}$ and $l_h(s) > 3l_h(\delta_1),$ we have that \[l_h(s) + \left|\tau_1\right|l_h(\delta_1) + \left|\tau_2\right|l_h(\delta_2) \geq l_h(\delta_1)(3 + \left|\tau_1\right| + \left|\tau_2 \right|) = \frac{9}{10}(3 + \left|\tau_1\right| + \left|\tau_2 \right|),\] and so by \eqref{eq:1} we have that $2l_h(a) \geq \frac{9}{10}(2 + \left|t_1\right| + \left|t_2 \right|) = \frac{9}{10}l_p(a)$, as required.

In order to prove $\eqref{eq:2}$ one considers three cases, depending on whether $\tau_1\tau_2$ is positive, negative or zero. All three follow from elementary hyperbolic geometry computations. We prove one of the three cases below, leaving the details of the remaining cases to the reader. 

\begin{figure}
\centering
  \includegraphics[scale = 0.35]{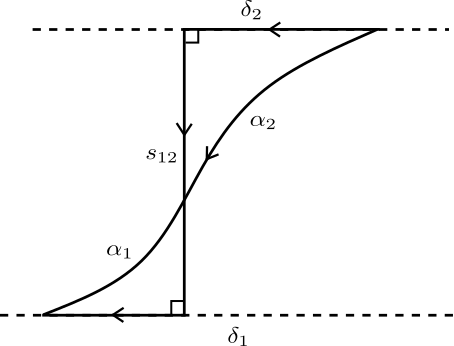}   
\caption{Schematic picture of lifts of the arcs $a$ and $a^{*}$ to the universal cover of a pair of pants}
  \label{lift}
\end{figure}

Assume $\tau_1\tau_2 > 0$, and choose lifts of the arcs $a, a^{*}$ to the universal cover of $P$ to be such that the endpoints of the lift of $a$ coincide with the endpoints of the lift of $a^{*}$. By abuse of notation, we write $a, a^{*}$ to also denote the lifts of the corresponding arcs. Since the seam $s_{12}$ intersects the boundary components at right angles, we have that $a, a^{*}$ form the sides of two right triangles. We split $a = a_1 + a_2$ into two sub-arcs, each of which is the hypotenuse of one of the triangles, see Figure~\ref{lift}. Using elementary result from hyperbolic geometry, we have that $l_h{(a_1)} \geq \left | \tau_1 \right | l_h{(\delta_1)} $ and $l_h{(a_2)} \geq \left | \tau_2 \right | l_h{(\delta_2)} .$ Furthermore, by definition of the seam we must have that $l_h{(a)\geq l_h{(s_{12})}  }.$ The bound in \eqref{eq:2} follows.

\end{proof}

\begin{claim}\label{same} Let $p \in  \{p_i\}_{i=1}^3$. The bound in \eqref{eq:0} holds for any simple loop $a$ in $P$ based at $p$.
\end{claim}

\begin{proof}

Let $\delta$ denote the boundary component of $P$ which contains the endpoints of $a$, and let $\nu$ be the mid of $P$ with endpoints in $\delta$, i.e. the the shortest essential arc joining $\delta$ to itself. Set $a^{*}$ to be the unique arc of shortest length which is homotopic to $a$ and which traverses only the boundary $\delta$ and the mid $\nu$. Since $a$ is simple, it must be that when $a^{*}$ traverses $\delta$ for the second time, it is travelling in the opposite direction to the first time. Let $\left| \tau^{+} \right|, \left| \tau^{-} \right|$ be the length of the subarc of $a^{*}$ which traverses $\delta$ in the positive and negative directions, respectively. Let $t$ denote the twisting parameter $a$ corresponding to the boundary component $\delta$. Clearly, \begin{equation} \left| \tau^{+} \right| + \left| \tau^{-} \right| \geq \left |t \right| - 2.\label{eq:3} \end{equation}

By lifting the arcs $a$ and $a^{*}$ to the universal cover of $P$ as in proof of Claim~\ref{diff}, we get that 
\begin{equation}
2l_h{(a)} \geq \frac{9}{10} (\left|\tau^{+} \right| + \left| \tau^{-} \right|  + 4) \label{eq:4},
\end{equation} 

this time using the fact that the length of the mid satisfies $l_h{(\nu)} \approx 4.57 \geq 4l_h(\delta).$ The required result follows by combining \eqref{eq:3} and \eqref{eq:4}.

\end{proof}

The generalisation of Lemma \ref{6.1} to multiarcs in $P$ follows directly by the definition of Dehn-Thurston coordinates:

\begin{corollary}\label{6.3} If $C$ is a multiarc in P with endpoints coinciding with the marked points  $\{p_i\}_{i=1}^3 \subset \partial P$, then \[ l_p{(C)} \leq \frac{20}{9}l_h{(C)}. \] 
\end{corollary}

\subsection{Proof of Proposition~\ref{comb length}}

Let $\mathcal{P} = \{P_i\}_{i=1}^M$ denote the collection of pairs of pants in the pants decomposition $\mathcal{K} = \{K_1, \cdots, K_N\}$ of $\Sigma$ from before. Fix a complete hyperbolic metric $h$ on $\Sigma$ to be such that the length of each pants curve is $\frac{9}{10}$. Fix the set of marked points $\{p_i\}_{i=1}^N$ in the pants curves, and the set of canonical arcs connecting them, as before.

\begin{proof}[Proof of Proposition~\ref{comb length}]

Let $\alpha$ be a simple geodesic curve. For each $j$ such that $\iota(\alpha, K_j) \neq 0$, homotope $\alpha$ in a small neighbourhood of $K_j$ so that it intersects $K_j$ exactly at the marked point $p_j$, and so that the resulting curve only self-intersects at the marked points. Let $\alpha^{*}$ be the curve obtained via this homotopy, and for every $j$ let $\alpha_j = \alpha^{*} \cap P_j$. We define the \emph{pants length} of $\alpha^{*}$ to be \[l_{h,\mathcal{K}}(\alpha^{*}) \defeq \sum_{j=1}^Ml_h({\alpha_j}),\]
where each $l_h(\alpha_j)$ is understood to be the hyperbolic length of the multiarc $\alpha_j$ in $P_j$. We aim to find a constant $c > 0$ such that $l_{h,\mathcal{K}}(\alpha^{*}) \leq c\, l_h(\alpha)$. 

By the triangle inequality, $l_h(\alpha_j) \leq l_h(\alpha \cap P_j) + l_h(K)\iota(\alpha, \partial P_j)$ for every $j$, where $K$ is any pants curve in $\mathcal{K}$ (the pants curves all have the same length). Also $\sum_{j=1}^M \iota(\alpha, \partial P_j) = 2\iota(\alpha, \mathcal{K}),$ and $\sum_{j=1}^Ml_h(\alpha \cap P_j) \leq l_h(\alpha)$, since $\alpha$ is a geodesic. Hence  $l_{h,\mathcal{K}}(\alpha^{*}) \leq l_h(\alpha) + 2l_h(K)\iota(\alpha, \mathcal{K})$. By the Collar Lemma, there exists a constant $w(K) = \text{arcsinh}(1 / \sinh(\frac{l_h(K)}{2}))$, such that we can embed an annulus of width $2w$ around every pants curve in $\Sigma,$ with the property that the annuli are pairwise disjoint. Thus, at each intersection of $\alpha$ with some pants curve $K_j$, we must have that $\alpha$ traverses at least the width of the annular neighbourhood around $K_j$. Hence, we have that $\iota{(\alpha, \mathcal{K})} \leq \frac{l_h{(\alpha)}}{2w}$. Putting everything together,
\begin{equation}l_{h, \mathcal{K}}{(\alpha^{*})} \leq l_h{(\alpha)} \left(1 + \frac{l_h(K)}{w}\right) \leq \frac{8}{5}l_h(\alpha),
\label{eq:5}
\end{equation}
where the second inequality follows from noting that $l_h(K) = \frac{9}{10}$, so $w(K) = \text{arcsinh}(1 / \sinh(\frac{l_h(K)}{2})) \geq 3/2$ and thus $1 + \frac{l_h(K)}{w} \leq \frac{8}{5}$.

Finally, we relate the combinatorial length of $\alpha$ to the sum of the combinatorial lengths of the multiarcs $\alpha_j \subset P_j$ for $1 \leq j \leq M$. Let $p(\alpha_j) = (m^j_1, m^j_2, m^j_3) \times (t^j_1, t^j_2, t^j_3)$ be the Dehn-Thurston coordinate for the multiarc $\alpha_j$ in $P_j$. If we cut $\alpha^{*}$ and consider the intersections of the multiarcs $\{\alpha_1, \cdots, \alpha_M\}$ with the boundaries of the pairs of pants they're contained in, each intersection of $\alpha^{*}$ with a pants curve in $\mathcal{K}$ gives rise to exactly two intersections, and conversely every intersection of $\alpha_i$ with the boundary of a pair of pants arises in this way. (Note that this is because $\alpha^{*}$ does not intersect the boundary curves of $\Sigma$.) Furthermore, suppose two pairs of pants $P_{j}, P_{k}$ intersect at a common boundary which corresponds to the pants curve $K_i$, and $t_i$ is the twisting parameter of $\alpha$ around  $K_i$. Take $\alpha_j \subset P_{j}, \alpha_k \subset P_k$, and let $t_j, t_k$ be their respective twisting parameters around the pants legs corresponding to $K_i.$ Then the twisting parameters satisfy $ \left|t_i\right| = \left|t_j + t_k \right| \leq \left|t_j \right| + \left|t_k \right|$.  It follows that $l_p(\alpha^{*}) \leq \sum_{i=1}^Ml_p(\alpha_j)$. 

By the above remarks and Corollary~\ref{6.3},
\[\l_p{(\alpha)} \leq \sum_{l=1}^M{l_p{(\alpha_l})} \leq \frac{20}{9}\sum_{l=1}^Ml_h{(\alpha_l)} = \frac{20}{9}l_{h,\mathcal{K}}(\alpha^{*}).\] Combining this with \eqref{eq:5}, we get that
\[l_p{(\alpha)} \leq \frac{20}{9}l_{h,\mathcal{K}}(\alpha^{*}) \leq 4l_h(\alpha).\]

\end{proof}


\section{Algorithm for filling curves}\label{4}
\subsection{Filling curves}
Recall that a curve $\gamma \subset \Sigma$ in minimal position is \emph{filling}, if the components of $\Sigma - \gamma$ are topological disks and annuli, such that each annulus is homotopic to a boundary component of $\Sigma$. Equivalently, $\gamma$ is filling if and only if it intersects every essential simple curve in $\Sigma$. In fact the following stronger result holds, whose proof we include below for completeness.

\begin{lemma}\label{length}
Fix a hyperbolic metric $h$ on $\Sigma$, and let $\gamma$ be a non-peripheral closed geodesic in $\Sigma$. Then, the geodesic $\gamma$ is filling if and only if it intersects every essential simple closed curve $\alpha$ in $\Sigma$, with $l_h{(\alpha)} \leq 2 \ell_h{(\gamma)}.$
\end{lemma}

We define an \emph{essential subsurface of a curve $\gamma$}, denoted $\Sigma_{\gamma}$, to be the smallest subsurface of $\Sigma$ which contains $\gamma$, such that every component of $\partial \Sigma_{\gamma}$ is either contained in $\partial \Sigma$, or is an essential, simple curve in $\Sigma$.

\begin{proof}[Proof of Lemma~\ref{length}]
The forward direction is clear.

For the other direction, let $\gamma$ be a closed geodesic in $\Sigma$ and suppose $\gamma$ does not fill $\Sigma$. 
Let $\{\gamma_1, \cdots, \gamma_k\}$ be the geodesic boundary curves of the essential surface $\Sigma_{\gamma}$. We claim that $\sum_{i=1}^kl_h{(\gamma_i)} \leq 2l_h(\gamma)$. Indeed, since $\gamma$ fills $\Sigma_{\gamma}$ the complement $\Sigma_{\gamma} - \gamma$ is a set of pairwise-disjoint disks and annuli. Each $\gamma_i$ acts as a boundary component of exactly one annulus in the decomposition, whilst the other boundary is a concatenation of segments of $\gamma$ which are homotopic to $\gamma_i$. The segments of $\gamma$ can act as the boundary of at must two annuli, and thus the bound of the claim follows. Since $\gamma$ does not intersect any of the curves in $\{\gamma_1, \cdots, \gamma_k\}$, the result follows from the claim.

\end{proof}

\subsection{Algorithm for curve intersection}\label{7.2}

By Lemma~\ref{length}, in order to determine whether a curve $\gamma$ is filling, one needs to compute the intersection number of $\gamma$ with a finite collection of curves. There exists a number of algorithms for computing intersection numbers, taking as input curves in various combinatorial representations. The work of Tan \cite{Tan}, and Cohen and Lustig \cite{CohLus} gives algorithms for curves represented as words in a generating set of the fundamental group, for surfaces with nonempty boundary. The latter algorithm was extended by Lustig \cite{Lus} to also deal with the closed surface case. 

More recently, Despr\'e and Lazarus \cite{DesLaz} have constructed another such algorithm, which is of particular interest to us as it gives estimates for its running time. Given two curves represented as walks of length at most $L$ in an embedded graph in the surface $\Sigma$, the algorithm computes their intersection number in $O(L^2)$ time. We note that given our generating set $X$ (see Section~\ref{dict}), we can construct an embedded graph in $\Sigma$ in the following way. The set $X$ gives rise to an immersed graph with a single vertex $p$, and an edge for each generator. Homotoping each generator curve (fixing base point $p$) so that the curves are in minimal position, we add a vertex at each intersection point. Now each generator in $X$ corresponds to a walk of length at most $c$, where $c$ is some fixed constant depending only on the complexity of the surface. Thus a word in $X$ of length bounded by $L$ corresponds to a closed walk of length bounded by $cL$. 

We summarise the preceding discussion with the following theorem:

\begin{theorem}[Cohen-Lustig \cite{CohLus}, Lustig \cite{Lus}, Tan \cite{Tan}, Despr\'e-Lazarus \cite{DesLaz}]\label{intersections} Let $\Sigma$ be a surface of negative Euler characteristic. There exists an algorithm to determine whether two curves represented as words have non-zero geometric intersection. Furthermore, if the words which represent the curves have length at most $L$, the algorithm terminates in $O(L^2)$ time.
\end{theorem}

\subsection{Proof of the main result}\label{7.3}

We now prove the main results of the paper. Along the way we also prove Proposition~\ref{list}.

\begin{proof}[Proof of Theorem \ref{main}]
Fix a pants decomposition of $\Sigma$, and a complete hyperbolic metric $h$ where each pants curve has length $\frac{9}{10}$. Fix the generating set $X$ of $\pi_1(\Sigma)$, as before. Let $\gamma$ be a curve in $\Sigma$, represented as a word $x_{\gamma}$ in $X$ of length $L$. From the calculations in Section~\ref{dict}, $l_h(\gamma) \leq 26ML = L'$, where $M = 2g - 2 +n $ is the number of pairs of pants in $\Sigma$.

Let $\mathcal{C} = \mathcal{C}(8L')$ denote the set of Dehn-Thurston coordinates of curves of combinatorial length bounded by $8L'$. By Theorem~\ref{comb length}, $\mathcal{C}$ contains all simple curves of hyperbolic length bounded by $2L'$. Using the dictionary given in Section~\ref{dict}, translate the Dehn-Thurston coordinates into words in $X$, and let $\mathcal{W}(L')$ denote the resulting set of words. Using Theorem~\ref{intersections}, one checks the geometric intersection number of $x_{\gamma}$ with each of the words in $\mathcal{W}(L')$.  If there exists a word in $\mathcal{W}(L')$ which does not intersect $x_{\gamma}$, then by Lemma~\ref{length} $\gamma$ is not filling. Otherwise, $\gamma$ is filling. 

To see that this procedure terminates in polynomial time, note that $\mathcal{W}(L')$ contains at most $2^N{8L' + 2N \choose 2N}  = O(L'^{2N}) = O(L^{2N})$ words, where $N = 3g - 3 + n$ is the number of pants curves in $\Sigma$. Since the intersection algorithm gives us a running time of $O(L^2)$, our algorithm terminates in $O(L^{2N + 2})$ steps.

\end{proof}


\bibliographystyle{alpha}
\bibliography{FillingCurvesAlgorithm}

\end{document}